\documentclass[12pt,a4paper]{article}
\usepackage{amsmath}
\usepackage{amsfonts,dsfont}
\usepackage{amssymb}
\usepackage{setspace}
\usepackage{amsthm}
\usepackage{rotating}
\usepackage[a4paper, left=2.5cm, right=2.5cm, top=2.8cm, bottom=2.8cm]{geometry}
\usepackage{float}
\usepackage{amsmath}
\usepackage{graphicx}
\usepackage{subfigure,bm}
\usepackage{natbib}
\usepackage{marginnote} 

\usepackage{dsfont}

\newcommand{\bol}[1]{\mbox{\boldmath$#1$}}
\newcommand{\bmu}{\bol{\mu}}

\newcommand{\bnu}{\bol{\nu}}

\newcommand{\bSigma}{\mathbf{\Sigma}}
\newcommand{\bOmega}{\mathbf{\Omega}}
\newcommand{\bM}{\mathbf{M}}

\newtheorem{theorem}{Theorem}
\newtheorem{proposition}{Proposition}

\newtheorem{corollary}{Corollary}

\usepackage{natbib}
\setcitestyle{numbers,square}
\begin{document}
\begin{center}
\vspace*{2cm} \noindent {\bf \large On the product of a singular Wishart matrix and a singular Gaussian vector in high dimension}\\
\vspace{1cm} \noindent {\sc Taras Bodnar$^{a,}$\footnote{Corresponding author. E-mail address: taras.bodnar@math.su.se. The first and the third authors appreciate the financial support of SIDA via the project 1683030302. The second author gratefully acknowledges financial support from the research project "Ambit Fields: Probabilistic
properties and statistical inference" funded by Villum Fonden.}, Stepan Mazur$^{b}$, Stanislas Muhinyuza$^{a,c}$, Nestor Parolya$^{d}$
}\\
\vspace{1cm}
{\it \footnotesize  $^a$ Department of Mathematics, Stockholm University, Roslagsv\"{a}gen 101, SE-10691 Stockholm, Sweden}\\
{\it \footnotesize  $^b$ Department of Mathematics, Aarhus University, Ny Munkegade 118, DK-8000 Aarhus, Denmark} \\
{\it \footnotesize $^c$Department of Mathematics, College of Science and Technology,University of Rwanda, P.O. Box 3900, Kigali-Rwanda}\\
{\it \footnotesize  $^d$ Institute of Statistics, Leibniz University of Hannover, D-30167 Hannover, Germany} \\
\end{center}

\vspace{1cm}
\begin{abstract}
In this paper we consider the product of a singular Wishart random matrix and a singular normal random vector. A very useful stochastic representation is derived for this product, using which the characteristic function of the product and its asymptotic distribution under the double asymptotic regime are established.
The application of obtained stochastic representation speeds up the simulation studies where the product of a singular Wishart random matrix and a singular normal random vector is present. We further document a good performance of the derived asymptotic distribution within a numerical illustration. Finally, several important properties of the singular Wishart distribution are provided.
\end{abstract}

\vspace{1cm}
\noindent ASM Classification: 60E05, 60E10, 60F05, 62H10, 62E20 

\newpage
\section{Introduction}
The multivariate normal distribution is one of the basic central distributions and a building block in multivariate statistical analysis. It is also a standard assumption in many statistical applications where the normal distribution is usually accompanied by the Wishart distribution. Namely, if we consider a sample of size $n$ from $k$-dimensional normal distribution, the unbiased estimators for the mean vector and for the covariance matrix have a $k$-dimensional normal distribution and a $k$-dimensional Wishart distribution, respectively, as well as they are independent (see, e.g., \citet[Chapter 3]{Muirhead1982}).

A number of papers deal either with the properties of the sample mean vector or with the properties of the sample covariance matrix, although these two objects appear often together in the expressions of different statistics. Consequently, a question arises how the distributions of functions involving both a Wishart matrix and a normal vector could be derived. Recently, this topic has attracted a lot of attention in the literature from both the theoretical perspectives (c.f., \citet{bodnar2011product,BodnarMazurOkhrin2013}) and the applications (see, e.g., \citet{jobson1980estimation,kan2007optimal, bodnar2009sequential}). While \citet{bodnar2011product}, \citet{kotsiuba2015asymptotic} derived the exact distribution and the approximative distribution of the product of an inverse Wishart matrix and a normal vector, \citet{BodnarMazurOkhrin2013} extended these results to the case of the product of a Wishart matrix and a normal vector. It is remarkable that the product of an inverse Wishart matrix and a normal vector has a direct application in discriminant analysis (c.f., \citet{rencher2012wiley}) and in portfolio theory (see, e.g., \citet{britten1999sampling}). On the other hand, the product of a Wishart matrix and a normal vector can be considered from the viewpoint of Bayesian statistics when inferring the discriminant function and the portfolio weights by employing the inverse Wishart - normal prior which appears to be a conjugate prior for the mean vector and the covariance matrix under normality (see, e.g., \citet{BernardoSmith1994}).

Singular covariance matrix is present in practical applications, especially when data are drawn from a large-dimensional process. For example, the construction of an optimal portfolio with a singular covariance matrix has become an important topic in finance (see, e.g., \citet{PappasKiriakopoulosKaimakamis2010, BodnarMazurPodgorski2015}). While the normal distribution with the singular covariance matrix is known as the singular normal distribution in statistical literature, there is no unique definition in the case of the Wishart distribution. The singular Wishart distribution introduced by \citet{khatri1970note} and \citet{srivastavaintroduction} deals with the case when the sample size is smaller than the process dimension. The practical relevance of this case is discussed in \citet{uhlig1994singular}, while some theoretical finding were derived in \citet{Srivastava2003}, \citet{BodnarOkhrin2008}, \citet{BodnarMazurOkhrin2014}. Another type of the singular Wishart distribution, the so-called pseudo-Wishart distribution, was discussed in \citet{diaz1997wishart}. The covariance matrix is assumed to be singular in this family of Wishart distributions. Later on, we refer to the Wishart distribution with a singular covariance matrix independently if the sample size is larger or smaller than the process dimension as the singular Wishart distribution.

In the present paper, we contribute to the existent literature on the singular Wishart distribution and the singular normal distribution by considering both distributions not separately but in a combination. We derive a very useful stochastic representation for the product of a singular Wishart matrix and a normal vector which provides an elegant way of characterizing the finite sample distribution of the product as well as it appears to be very useful in the derivation of its asymptotic distribution under the high-dimensional asymptotic regime, i.e. when both the sample size and the process dimension become very large.

The rest of the paper is structured as follows. Section 2 contains several distributional properties of the singular Wishart distribution which are used as a tool to prove the main results of the paper presented in Section 3. Here, the distribution of the product of a singular Wishart matrix and a singular normal random vector is derived in terms of a stochastic representation from which we also obtain the characteristic function of the product. Furthermore, we prove the asymptotic normality of the product under the high-dimensional asymptotic regime. The finite sample performance of the obtained asymptotic results is discussed in Section 4, while Section 5 presents the summary.


\section{Preliminary results}
\label{s2}

In this section, we present several distributional properties of the singular Wishart distribution which are used in proving the main results of the paper.

Let $\mathbf A \sim \mathcal W _k (n, \bSigma )$, i.e., the random matrix $\mathbf A$ has a $k$-dimensional singular Wishart distribution with $n$ degrees of
freedom and covariance matrix $\bSigma$ which is positive semi-definite with $rank (\bSigma) = r < k$. Throughout the paper, no assumption is made about the relationship between the sample size $n$ and the process dimension $k$. The results are valid in both cases $n \ge k$ (Wishart distribution with positive semi-definite covariance matrix $\bSigma$) and $k<n$ (singular Wishart distribution with positive semi-definite covariance matrix $\bSigma$). Also, let $\mathbf I _k $ be the $k\times k$ identity matrix, $\otimes$ stands for the Kronecker product, and the symbol $\stackrel{d}{=}$ denotes the equality in distribution.

In Proposition \ref{th1}, we derive the distribution of a linear symmetric transformation of the singular Wishart random matrix.


\begin{proposition}
\label{th1}
Let $\mathbf A \sim \mathcal W _{k} (n, \bSigma)$ with $rank (\bSigma)=r < k$ and let $\mathbf M: p\times k$ be a matrix of constants with $rank(\bM)=p$ such that $\mathbf M \bSigma \neq \mathbf 0 $.
Then
 $$\mathbf {MAM} ^T \sim \mathcal W _{p} (n, \mathbf {M \Sigma M} ^T) .$$

 Moreover, if $rank (\mathbf M \bSigma)=p \leq r $, then $\mathbf {MAM} ^T$ and $\mathbf {M\Sigma M} ^T$ are of the full rank $p$.
\end{proposition}


\begin{proof}
From Theorem 5.2 of \citet{Srivastava2003} we have that the stochastic representation of $\mathbf A$ is given by
\begin{eqnarray*}
\mathbf A \stackrel{d}{=} \mathbf {XX} ^T\ \ \ \ \text{with}\ \ \ \ \mathbf X \sim \mathcal N _{k,n} (\mathbf 0, \bSigma\otimes \mathbf I _n).
\end{eqnarray*}

Then using Theorem 2.4.2 of \citet{GuptaNagar2000} we get
\begin{eqnarray*}
\mathbf {MAM} ^T \stackrel{d}{=} \mathbf {MXX} ^T \mathbf M^T\stackrel{d}{=} \mathbf {YY}^T,
\end{eqnarray*}
where $\mathbf Y \sim \mathcal {N} _{p,n} (\mathbf 0, (\mathbf {M\Sigma M}^T) \otimes \mathbf I_n)$.
This completes the proof of the proposition.
\end{proof}


An application of Proposition \ref{th1} leads the following result summarized in Proposition \ref{th2}.

\begin{proposition}
\label{th2}
Let $\mathbf A \sim \mathcal W _{k} (n, \bSigma)$ with $rank (\bSigma)=r < k$ and let $\mathbf W: p\times k$ be a random matrix which is independent of $\mathbf A$ such that $rank (\mathbf W \bSigma)=p \leq r \leq n $ with probability one. Then
\begin{eqnarray*}
(\mathbf W \bSigma \mathbf W ^ T )^{-1/2} (\mathbf W \mathbf A \mathbf W ^T) (\mathbf W \bSigma \mathbf W ^ T )^{-1/2} \sim \mathcal W _p (n, \mathbf I_p)
\end{eqnarray*}
and is independent of $\mathbf W$.
\end{proposition}


\begin{proof}
Using the fact that $\mathbf W$ and $\mathbf A$ are independently distributed,  we obtain that the conditional distribution of $\mathbf {WAW} ^T | (\mathbf W=\mathbf W_0)$ is equal to the distribution of $\mathbf W_0 \mathbf A \mathbf W_0^T$.
Then applying Proposition~\ref{th1} we obtain
\begin{eqnarray*}
(\mathbf W_0 \bSigma \mathbf W_0 ^ T )^{-1/2} (\mathbf W_0 \mathbf A \mathbf W _0^T) (\mathbf W_0 \bSigma \mathbf W_0 ^ T )^{-1/2} \sim \mathcal W _p (n, \mathbf I_p).
\end{eqnarray*}
Since this distribution does not depend on $\mathbf W$, it is also the unconditional distribution of $(\mathbf W \bSigma \mathbf W ^ T )^{-1/2} (\mathbf W \mathbf A \mathbf W ^T) (\mathbf W \bSigma \mathbf W ^ T )^{-1/2}$. The proposition is proved.
\end{proof}


In the next corollary, we consider a special case of Proposition \ref{th2} with $p=1$.

\begin{corollary}
\label{c1}
Let $\mathbf A \sim \mathcal W _{k} (n, \bSigma)$ with $rank (\bSigma)=r \leq k$ and let $\mathbf w$ be a $p$-dimensional vector which is independent of $\mathbf A$ with $P (\mathbf w^T \bSigma = \mathbf{0})=0$. Then
\begin{eqnarray*}
\frac{\mathbf w ^T \mathbf {Aw} } {\mathbf w ^T \bSigma \mathbf w} \sim \chi^2_{n},
\end{eqnarray*}
and is independent of $\mathbf w$.
\end{corollary}


\section{Main results}
\label{s3}

In this section, we present the main results of the paper which are complementary to the ones obtained in \citet{BodnarMazurOkhrin2013,BodnarMazurOkhrin2014} to the case of high-dimensional data and singular covariance matrix.


\subsection{Finite sample results}

Let $\mathbf z$ be $k$-dimensional singular normally distributed random vector with mean vector $\bmu$ and covariance matrix $\kappa \bSigma$, $\kappa >0$, such that $rank (\bSigma) = r < k$, i.e. $\mathbf z \sim \mathcal N _k (\bmu , \kappa \bSigma)$. Also, let $\mathbf M$ be a $p\times k$ matrix of constants with $rank (\mathbf M)=p \leq r \leq \min\{n,k\}$ such that $ \mathbf {M \Sigma} \neq \mathbf 0$. We are interested in the distribution of $\mathbf {MAz}$ when $\mathbf A$ and $\mathbf z$ are independently distributed where $\mathbf A$ has a singular Wishart distribution as defined in Section \ref{s2}.

In Theorem \ref{th3}, we derive a stochastic representation for $\mathbf {MAz}$. The stochastic representation is an tool in the theory of multivariate statistics and it is frequently used in Monte Carlo simulations (c.f., \citet{givens2012computational}). Its importance in the theory of elliptically contoured distributions is well described by \citet{GuptaVargaBodnar2013}.


\begin{theorem}
\label{th3}
Let $\mathbf A \sim \mathcal W _{k} (n, \bSigma)$ with $rank (\bSigma)=r < k$ and let $\mathbf z \sim \mathcal N _k (\bmu, \kappa \bSigma)$, $\kappa>0$.
We assume that $\mathbf A$ and $\mathbf z$ are independently distributed. Also, let $\mathbf M : p \times k$ be a matrix of constants of rank $p<r\le n$ and denote $\mathbf Q = \mathbf P^T \mathbf P$ with $\mathbf P= (\mathbf {M\Sigma M} ^T)^{-1/2} \mathbf M \bSigma^{1/2}$. Then the stochastic representation of $\mathbf {MAz}$ is given by
\begin{eqnarray*}
\mathbf {MAz}
&\stackrel{d}{=}&
\zeta \mathbf {M\Sigma}^{1/2} \mathbf t
+
\sqrt{\zeta}
(\mathbf {M\Sigma M}^T)^{1/2}
\left[
\sqrt{\mathbf t ^T  \mathbf t} \mathbf I _p
-
\frac {\sqrt{\mathbf t ^T \mathbf t} - \sqrt{\mathbf t ^T (\mathbf I_k -  \mathbf Q )\mathbf t}}  {\mathbf t ^T  \mathbf Q  \mathbf t}
\mathbf P \mathbf {tt}^T \mathbf P^T
\right]
\mathbf z_0,
\end{eqnarray*}
where $\zeta \sim \chi^2_n$, $\mathbf t \sim \mathcal N _k (\bSigma^{1/2}\bmu, \kappa \bSigma^2)$, and $\mathbf z_0\sim \mathcal N _p (\mathbf 0, \mathbf I _p)$;
$\zeta$, $\mathbf t$, and $\mathbf z_0$ are mutually independent.
\end{theorem}


\begin{proof}
Since $\mathbf A$ and $\mathbf z$ are independently distributed it holds that the conditional distribution of $\mathbf {MAz} | (\mathbf z = \mathbf z ^*)$ is equal to the distribution of $\mathbf {MAz} ^*$.

Let $\widetilde{ \mathbf M}$ be the matrix which is obtained from $\mathbf M$ by adding a row vector $\mathbf z^*$, i.e. $\widetilde{ \mathbf M} = (\mathbf M^T, \mathbf z^*)^T$. Consider the following two partitioned matrices
\begin{eqnarray*}
\widetilde {\mathbf A}
=
\widetilde {\mathbf M} \mathbf A \widetilde {\mathbf M}^T
=
\left(
\begin{array}{cc}
\mathbf {MAM}^T & \mathbf {MAz}^* \\
\mathbf z ^{*T} \mathbf {A M}^T & \mathbf z ^{*T} \mathbf {A z} ^*
\end{array}
\right)
=
\left(
\begin{array}{cc}
\widetilde {\mathbf A}_{11} & \widetilde {\mathbf A}_{12}\\
\widetilde {\mathbf A}_{21} & \widetilde {A}_{22}
\end{array}
\right)
\end{eqnarray*}
and
\begin{eqnarray*}
\widetilde {\mathbf \Sigma}
=
\widetilde {\mathbf M} \mathbf \Sigma \widetilde {\mathbf M}^T
=
\left(
\begin{array}{cc}
\mathbf {M\Sigma M}^T & \mathbf {M\Sigma z}^* \\
\mathbf z ^{*T} \mathbf {\Sigma M}^T & \mathbf z ^{*T} \mathbf {\Sigma z} ^*
\end{array}
\right)
=
\left(
\begin{array}{cc}
\widetilde {\mathbf \Sigma}_{11} & \widetilde {\mathbf \Sigma}_{12}\\
\widetilde {\mathbf \Sigma}_{21} & \widetilde {\Sigma}_{22}
\end{array}
\right)
\end{eqnarray*}

Since $\mathbf A \sim \mathcal W _k (n,\bSigma)$ and $rank ( \widetilde {\mathbf M} ) = p+1\leq r $, it holds that $\widetilde {\mathbf A} \sim \mathcal W_{p+1} (n, \widetilde \bSigma)$ following Proposition \ref{th1}. Using Theorem 3.2.10 of \citet{Muirhead1982}, we get the conditional distribution of $\widetilde {\mathbf A} _{12}= \mathbf {MAz}^*$ given $\widetilde {A}_{22}$ expressed as
\begin{eqnarray*}
\widetilde {\mathbf A} _{12} | \widetilde A _{22}
\sim
\mathcal N _{p} \left( \widetilde \bSigma _{12}  \widetilde \Sigma _{22} ^{-1} \widetilde A _{22}, \widetilde \bSigma _{11\cdot 2} \widetilde A _{22}  \right)
\end{eqnarray*}
with $\widetilde \bSigma _{11\cdot 2} = \widetilde \bSigma _{11} - \widetilde \bSigma _{12} \widetilde \Sigma _{22}^{-1} \widetilde \bSigma _{21}$.

Let $\zeta = \widetilde {A} _{22}\widetilde \Sigma _{22}^{-1} $. Then from Corollary \ref{c1} we get that $\zeta \sim \chi^2_n$, and it is independent of $\mathbf z$. Hence,
\begin{eqnarray*}
\mathbf {MAz} | \zeta, \mathbf z
\sim
\mathcal {N} _{p}
\left(
\zeta \mathbf {M\Sigma z},
\zeta (\mathbf z ^T \bSigma \mathbf z \mathbf {M\Sigma M}^T - \mathbf {M\Sigma z z} ^T \bSigma \mathbf M ^T)
\right),
\end{eqnarray*}
which leads to the stochastic representation of $\mathbf {MAz}$ given by
\begin{equation}\label{th3_eq1}
\mathbf {MAz}\stackrel{d}{=} \zeta \mathbf {M\Sigma z}
+\sqrt{\zeta}(\mathbf z ^T \bSigma \mathbf z \mathbf {M\Sigma M}^T - \mathbf {M\Sigma z z} ^T \bSigma \mathbf M ^T)^{1/2}\mathbf z_0,
\end{equation}
where $\zeta \sim \chi^2_n$, $\mathbf z \sim \mathcal {N}_k(\bmu, \kappa \bSigma)$, and $\mathbf z_0 \sim \mathcal N _p (\mathbf 0, \mathbf I _p) $.
Moreover, $\zeta$, $\mathbf z$, and $\mathbf z_0$ are mutually independent.

Next, we calculate the square root of $(\mathbf z ^T \bSigma \mathbf z \mathbf {M\Sigma M}^T - \mathbf {M\Sigma z z} ^T \bSigma \mathbf M ^T)$ using the following equality
\begin{eqnarray*}
(\mathbf D - \mathbf {bb}^T )^{1/2} = \mathbf D ^{1/2} (\mathbf I_p - c \mathbf D ^{-1/2} \mathbf {bb}^T \mathbf {D} ^{-1/2} )
\end{eqnarray*}
with $c=\frac{1- \sqrt{1-\mathbf b ^T \mathbf D^{-1} \mathbf b}}{\mathbf b ^T \mathbf D^{-1} \mathbf b}$, $\mathbf b = \mathbf M \bSigma \mathbf z$, and $\mathbf D = \mathbf z ^T \bSigma \mathbf z \mathbf {M\Sigma M}^T $ that leads to
\begin{eqnarray*}
\mathbf {MAz}
&\stackrel{d}{=}&
\zeta \mathbf {M\Sigma z}
+
\sqrt{\zeta}
(\mathbf {M\Sigma M}^T)^{1/2}\\
&&\times
\left[
\sqrt{\mathbf z ^T \bSigma \mathbf z} \mathbf I _p
-
\frac {\sqrt{\mathbf z ^T \bSigma \mathbf z} - \sqrt{\mathbf z ^T (\bSigma - \bSigma^{1/2} \mathbf Q \bSigma ^{1/2})\mathbf z}}  {\mathbf z ^T \bSigma^{1/2} \mathbf Q \bSigma^{1/2} \mathbf z}
\mathbf P \bSigma ^{1/2} \mathbf {zz}^T \bSigma ^{1/2} \mathbf P^T
\right]
\mathbf z_0,
\end{eqnarray*}
where $\mathbf P = (\mathbf {M\Sigma M}^T)^{-1/2} \mathbf M \bSigma ^{1/2}$ and $\mathbf Q = \mathbf P^T \mathbf P$.

Finally, making the transformation $\mathbf t = \bSigma ^{1/2} \mathbf z \sim \mathcal {N}_{k} (\bSigma ^{1/2}\bmu, \kappa \bSigma^2)$, we obtain the statement of the theorem.
\end{proof}

Next, we consider the special case of Theorem \ref{th3} when $p=1$ and $\mathbf M= \mathbf m^T$.

\begin{corollary}\label{c2}
Let $\mathbf A \sim \mathcal W _{k} (n, \bSigma)$ with $rank (\bSigma)=r < k$ and let $\mathbf z \sim \mathcal N _k (\bmu, \kappa \bSigma)$, $\kappa>0$.
We assume that $\mathbf A$ and $\mathbf z$ are independently distributed. Let $\mathbf m$ be a $k$-dimensional vector of constants such that $\mathbf{m}^T\bSigma \mathbf{m}>0$. Then the stochastic representation of $\mathbf m ^T \mathbf {Az}$ is given by
\begin{equation}\label{c2_sp}
\mathbf m ^T \mathbf {Az}
\stackrel{d}{=}
\zeta \mathbf m^T \bSigma \mathbf z  +
\sqrt{\zeta}
\left[
\mathbf z^T \bSigma \mathbf z \cdot \mathbf m^T \bSigma \mathbf m - (\mathbf m^T \bSigma \mathbf z)^2
\right]^{1/2} z_0,
\end{equation}
where $\zeta \sim \chi^2_n$ and $z_0 \sim \mathcal N (0,1)$; $\zeta$, $z_0$, and $\mathbf z$ are mutually independent.
\end{corollary}

The proof of Corollary \ref{c2} follows directly from \eqref{th3_eq1}. The result of the corollary is very useful from the viewpoint of computational statistics. Namely, in order to get a realization of $\mathbf m ^T \mathbf {Az}$ it is sufficient to simulate two random variables from the standard univariate distributions together with a random vector which has a singular multivariate normal distribution. There is no need to generate a large-dimensional object $\mathbf{A}$ and, as a result, the application of \eqref{c2_sp} speeds up the simulations where the product of $\mathbf{A}$ and $\mathbf{z}$ is present.

Another application of Corollary \ref{c2} leads to the expression of the characteristic function of $\mathbf {Az}$ presented in the following theorem.
\begin{theorem}
\label{th4}
Let $\mathbf A \sim \mathcal W _{k} (n, \bSigma)$ with $rank (\bSigma)=r < k$ and let $\mathbf z \sim \mathcal N _k (\bmu, \kappa \bSigma)$.
We assume that $\mathbf A$ and $\mathbf z$ are independently distributed.
Then the characteristic function of $\mathbf {Az}$ is given by
\begin{eqnarray*}
\varphi_{\mathbf {Az}} ( \mathbf u)&=&
\frac{ \exp \left( -\frac{\kappa^{-1}}{2} \bmu^T \mathbf R\mathbf \Lambda^{-1} \mathbf R^T\bmu\right)}{\kappa^{r/2} |\mathbf \Lambda|^{1/2}}
\int_{0}^{\infty} |\bOmega(\zeta)|^{-1/2}  f_{\chi^2_n} (\zeta)\\
&\times&\exp \left(
i \zeta \bnu^T \mathbf \Lambda \mathbf R^T \mathbf u - \frac{\zeta^2}{2} \mathbf u^T  \mathbf R \mathbf \Lambda \bOmega(\zeta)^{-1} \mathbf \Lambda \mathbf R^T \mathbf u + \frac{1}{2} \bnu^T \bOmega(\zeta) \bnu\right)
  \mbox{d} \zeta,
\end{eqnarray*}
where $\bnu  =  \kappa^{-1} \bOmega(\zeta)^{-1}  \mathbf \Lambda^{-1} \mathbf R^T \bmu$,
\begin{eqnarray*}
 \bOmega (\zeta)
 &=&\kappa^{-1} \mathbf \Lambda^{-1}+ \zeta \left[\mathbf \Lambda \cdot \mathbf u^T \mathbf \Sigma \mathbf u -\mathbf \Lambda \mathbf R^T \mathbf u \mathbf u^T \mathbf R \mathbf \Lambda
 \right],
\end{eqnarray*}
and
$\bSigma=\mathbf R\mathbf \Lambda \mathbf R^T$ is the singular value decomposition of $\bSigma$ with diagonal matrix $\mathbf \Lambda$ consisting of all $r$ non-zero eigenvalues of $\bSigma$ and $\mathbf R$ the $k \times r$ matrix of the corresponding eigenvectors; $f_{\chi^2_n}$ denotes the density function of the $\chi^2$ distribution with $n$ degrees of freedom.
\end{theorem}

\begin{proof}
From the stochastic representation derived in Corollary \ref{c2}, we get that
\begin{eqnarray*}
\varphi_{\mathbf {Az}} ( \mathbf u)&=&\mathbb{E}\left(\exp\left(i \mathbf {u}^T\mathbf{Az}\right)\right)
=\mathbb{E}\left(\exp\left( i \zeta \mathbf u^T \bSigma \mathbf z  +
i \sqrt{\zeta} \left[\mathbf z^T \bSigma \mathbf z \cdot \mathbf u^T \bSigma \mathbf u - (\mathbf u^T \bSigma \mathbf z)^2
\right]^{1/2} z_0\right)\right)\\
&=&\mathbb{E}\left(\exp\left(i \zeta \mathbf u^T \bSigma \mathbf z\right)
\mathbb{E}\left(\exp\left( i \sqrt{\zeta} \left[\mathbf z^T \bSigma \mathbf z \cdot \mathbf u^T \bSigma \mathbf u - (\mathbf u^T \bSigma \mathbf z)^2
\right]^{1/2} z_0\right)|\zeta,\mathbf{z}\right)\right)\\
&=&\mathbb{E}\left(\exp\left( i \zeta \mathbf u^T \bSigma \mathbf z\right)
\exp\left( - \frac{1}{2}\zeta \left[\mathbf z^T \bSigma \mathbf z \cdot \mathbf u^T \bSigma \mathbf u - (\mathbf u^T \bSigma \mathbf z)^2
\right]\right)\right)\\
&=&\mathbb{E}\left(\mathbb{E}\left(\exp\left( i \zeta \mathbf u^T \bSigma \mathbf z\right)
\exp\left(-\frac{1}{2}\zeta \left[\mathbf z^T \bSigma \mathbf z \cdot \mathbf u^T \bSigma \mathbf u - (\mathbf u^T \bSigma \mathbf z)^2
\right]\right)|\zeta\right)\right)\\
&=&\mathbb{E}\left(\mathbb{E}\left(\exp\left( i \zeta \mathbf v^T \mathbf \Lambda \mathbf y\right)
\exp\left( - \frac{1}{2}\zeta \left[\mathbf y^T \mathbf \Lambda \mathbf y \cdot \mathbf v^T \mathbf \Lambda \mathbf v - (\mathbf v^T \mathbf \Lambda \mathbf y)^2
\right]\right)|\zeta\right)\right)
\end{eqnarray*}
where $\mathbf v= \mathbf R ^T \mathbf u$; $\bSigma=\mathbf R\mathbf \Lambda\mathbf R^T$ is the singular value decomposition of $\bSigma$; $\mathbf y=\mathbf R^T\mathbf z \sim \mathcal{N}_r(\mathbf R^T \bmu, \kappa \mathbf \Lambda)$ has a non-singular multivariate normal distribution.

Hence,
\begin{eqnarray*}
&&\mathbb{E}\left(\exp\left( i \zeta \mathbf v^T \mathbf \Lambda \mathbf y\right)
\exp\left( - \frac{1}{2}\zeta \left[\mathbf y^T \mathbf \Lambda \mathbf y \cdot \mathbf v^T \mathbf \Lambda \mathbf v - (\mathbf v^T \mathbf \Lambda \mathbf y)^2
\right]\right)|\zeta\right)\\
&=&\frac{1}{(2\pi \kappa)^{r/2}|\mathbf \Lambda|^{1/2}}\int_{\mathbb{R}^r} \exp\left( i \zeta \mathbf v^T \mathbf \Lambda \mathbf y\right)
\exp\left( - \frac{1}{2}\zeta \left[\mathbf y^T \mathbf \Lambda \mathbf y \cdot \mathbf v^T \mathbf \Lambda \mathbf v - (\mathbf v^T \mathbf \Lambda \mathbf y)^2
\right]\right)\\
&\times&\exp\left(-\frac{\kappa^{-1}}{2}(\mathbf{y}-\mathbf R^T\bmu)^T \mathbf \Lambda^{-1} (\mathbf{y}-\mathbf R^T\bmu)\right) \mbox{d} \mathbf{y}
\end{eqnarray*}
where
\begin{eqnarray*}
 && \kappa^{-1}(\mathbf{y}-\mathbf R^T\bmu)^T \mathbf \Lambda^{-1} (\mathbf{y}-\mathbf R^T\bmu)+\zeta \left[\mathbf y^T \mathbf \Lambda \mathbf y \cdot \mathbf v^T \mathbf \Lambda \mathbf v - (\mathbf v^T \mathbf \Lambda \mathbf y)^2
 \right]\\
 &&=  (\mathbf{y}-\bnu)^T \bOmega(\zeta) (\mathbf{y}-\bnu) + d
\end{eqnarray*}
with
\begin{eqnarray*}
 \bOmega(\zeta) &=&\kappa^{-1} \mathbf \Lambda^{-1}+ \zeta \left[\mathbf \Lambda \cdot \mathbf v^T \mathbf \Lambda \mathbf v -\mathbf \Lambda \mathbf v\mathbf v^T \mathbf \Lambda \right] ,\\
 \bnu &=&   \kappa^{-1} \bOmega(\zeta)^{-1}  \mathbf \Lambda^{-1} \mathbf R^T \bmu, \\
 d &=&\kappa^{-1} \bmu^T \mathbf{R} \mathbf \Lambda^{-1} \mathbf R^T \bmu - \bnu^T \bOmega(\zeta) \bnu=
 \kappa^{-1} \bmu^T \bSigma^+ \bmu - \bnu^T \bOmega(\zeta) \bnu.
\end{eqnarray*}

As a result, we get
\begin{eqnarray*}
\varphi_{\mathbf {Az}} ( \mathbf u)&=&
\frac{ \exp \left( -\frac{\kappa^{-1}}{2} \bmu^T \bSigma^+ \bmu\right)}{\kappa^{r/2} |\mathbf \Lambda|^{1/2}}
\int_{0}^{\infty}|\bOmega(\zeta)|^{-1/2}  f_{\chi^2_n} (\zeta)
\\
&\times&\exp \left(i \zeta \bnu^T \mathbf \Lambda \mathbf v - \frac{\zeta^2}{2} \mathbf v^T \mathbf \Lambda \bOmega(\zeta)^{-1} \mathbf \Lambda \mathbf v + \frac{1}{2} \bnu^T \bOmega(\zeta) \bnu\right)
  \mbox{d} \zeta.
\end{eqnarray*}

This completes the proof of the theorem.
\end{proof}

\subsection{Asymptotic distribution under double asymptotic regime}

In this section we derive the asymptotic distribution of $\mathbf {MAz} $ under double asymptotic regime, i.e. when both $r$ and $n$ tend to infinity such that $r/n \to c \in [0,+\infty)$. In the derivation of the asymptotic distribution we rely on the results of Corollary \ref{c2}.

The following conditions are needed for ensuring the validity of the asymptotic results presented in this section
\begin{description}
\item[(A1)] Let $(\lambda_i,\bm u_i)$ denote the set of non-zero eigenvalues and eigenvectors of $\bm\Sigma$.
We assume that there exist $l_1$ and $L_1$ such that
\begin{eqnarray*}
0<l_1\leq\lambda_1\leq\lambda_2\leq\ldots\leq\lambda_r\leq L_1<\infty
\end{eqnarray*}
uniformly on $k$.
\item[(A2)] There exists $L_2$ such that
\begin{eqnarray*}
|\bm u_i^T\bmu|\leq L_2 \textrm{ for all } i=1,\ldots,r \textrm{ uniformly on } k.
\end{eqnarray*}
\end{description}

\begin{theorem}\label{th5}
Let $\mathbf A \sim \mathcal W _{k} (n, \bSigma)$ with $rank (\bSigma)=r < p$ and let $\mathbf z \sim \mathcal N _k (\bmu, \kappa \bSigma)$,$\kappa>0$. Assume $\frac{r}{n}=c+o(n^{-1/2}), c\in [0,+\infty)$ and $\kappa r =O(1)$ as $n\rightarrow \infty$. Also, let $\mathbf m$ be a $k$-dimensional vector of constants such that $\mathbf m^T \bSigma\mathbf m>0$ and $|\bm u_i^T\mathbf m|\leq L_2$ for all $i=1,\ldots,r$ uniformly on $k$. Assume that $\mathbf A$ and $\mathbf z$ are independently distributed. Then, under (A1) and (A2), it holds that the asymptotic distribution of $\mathbf m^T\mathbf{Az}$ is given by
\[\sqrt{n}\sigma^{-1}\left(\frac{1}{n}\mathbf m ^T \mathbf {Az}-\mathbf m ^T \bSigma\bmu\right) \stackrel{d}{\longrightarrow}\mathcal N \left(0,1\right), \]
where
\begin{eqnarray*}
\sigma^2&=& \left(\mathbf m^T \bSigma\bmu\right)^2+\mathbf m^T\bSigma\mathbf{m}
\left[\kappa tr(\bSigma^2)+\bmu^T \bSigma\bmu\right]+\frac{\kappa}{c}\mathbf m^T\bSigma^3\mathbf{m}.
\end{eqnarray*}
\end{theorem}

\begin{proof}
From Corollary \ref{c2}, the stochastic representation of $\mathbf m ^T \mathbf {Az}$ is given by
\begin{equation*}
\mathbf m ^T \mathbf {Az}
\stackrel{d}{=}
\zeta \mathbf m^T \bSigma \mathbf z  +
\sqrt{\zeta}
\left[
\mathbf z^T \bSigma \mathbf z \cdot \mathbf m^T \bSigma \mathbf m - (\mathbf m^T \bSigma \mathbf z)^2
\right]^{1/2} z_0,
\end{equation*}
with $\zeta \sim \chi^2_n$, $z_0 \sim \mathcal N (0,1)$ and $\mathbf{z}\sim\mathcal N _k (\bmu, \kappa \bSigma)$,$\kappa>0$; $\zeta$, $z_0$, and $\mathbf z$ are mutually independent.

From the property of $\chi^2$-distribution, we immediately obtain the asymptotic distribution of $\zeta$ given by
\begin{equation}\label{th5-Eq1}
\sqrt{n}\left(\frac{\zeta}{n}-1\right)\stackrel{d}{\sim}\mathcal N (0,2) \textrm{ as } n\rightarrow \infty
\end{equation}
Further, it holds that $\sqrt{n}(z_0/\sqrt{n})\sim \mathcal N (0,1)$ for all $n$, consequently it is its asymptotic distribution.

We next show that $\mathbf m ^T \mathbf {\bSigma z}$ and $\mathbf z ^T \mathbf {\bSigma z}$ are jointly asymptotically normally distributed under the high-dimensional asymptotic regime. For any $a_1\in \mathds{R}$ and $a_2 \in \mathds{R}$, we consider
\begin{equation*}
a_1\mathbf z ^T \mathbf {\bSigma z}+2a_2\mathbf m ^T \mathbf {\bSigma z}=
a_1\left(\mathbf z +\frac{a_2}{a_1}\mathbf m\right)^T\mathbf {\bSigma}\left(\mathbf z +\frac{a_2}{a_1}\mathbf m\right)
-\frac{a_2^2}{a_1}\mathbf m ^T \mathbf {\bSigma m}
=a_1\tilde{\mathbf z}^T \bSigma\tilde{\mathbf z}-\frac{a_2^2}{a_1}\mathbf m ^T \mathbf {\bSigma m}
\end{equation*}
where $\tilde{\mathbf z}\sim \mathcal N _k (\bmu_a, \kappa \bSigma)$ with $\bmu_a=\bmu+\frac{a_2}{a_1}\mathbf m$.
By \cite{provost1996exact} the random variable $\tilde{\mathbf z}^T \bSigma\tilde{\mathbf z}$ can be expressed as
\begin{equation*}
\tilde{\mathbf z}^T \bSigma\tilde{\mathbf z}
\stackrel{d}{=}
\kappa\sum_{i=1}^r\lambda_i^2\zeta_i \quad \textrm{with} \quad \zeta_i \stackrel{d}{\sim} \chi_1^2(\delta_i^2), ~\delta_i^2=\kappa^{-1}\lambda_i^{-1} \left(\bm u_i^T\bmu_a\right)^2,
\end{equation*}
where the symbol $\chi^2_d(\delta)$ denotes the non-central chi-squared distribution with $d$ degrees of freedom and non-centrality parameter $\delta$.

Next, we apply the Linderberg central limit theorem to the i.i.d random variables $V_i=\kappa\lambda_i^2\zeta_i$.
Let $\sigma_i^2=\mathbb{V}(V_i)$ and $s_n^2=\mathbb{V}(\sum_{i=1}^r V_i)$. It holds that
\begin{eqnarray*}
s_n^2&=&\mathbb{V}\left(\sum_{i=1}^rV_i\right)=\kappa^2 \sum_{i=1}^r\lambda_i^4\mathbb{V}(\zeta_i)=\kappa^2\sum_{i=1}^r\lambda_i^4 2(1+2\delta_i^2)
\\
&=&\kappa^2 \sum_{i=1}^r \left(2\lambda_i^4+4\kappa^{-1}\lambda_i^3(\bm u_i^T\bmu_a)^2\right)
=\kappa^2\left[2tr(\bSigma^4)+4\kappa^{-1}\bmu_a^T \bSigma^3\bmu_a\right]\,.
\end{eqnarray*}

In order to verify the Linderberg's condition, we need to check if for any small $\epsilon>0$ it holds that
\begin{equation}
\lim_{r\rightarrow\infty} \frac{1}{s_n^2} \sum_{i=1}^r \mathbb{E}\left[(V_i-\mathbb{E}(V_i))^2\mathds{1}_{\{|V_i-\mathbb{E}(V_i)|>\epsilon s_n\}}\right]= 0,
\end{equation}
where
\begin{eqnarray*}
&&\sum_{i=1}^r \mathbb{E}\left[(V_i-\mathbb{E}(V_i))^2\mathds{1}_{\{|V_i-\mathbb{E}(V_i)|>\epsilon s_n\}}\right]\\
&\stackrel{Cauchy-Schwarz}{\leq}&\sum_{i=1}^r\sqrt{\mathbb{E}\left[(V_i-\mathbb{E}(V_i))^4\right]} \sqrt{\mathbb{E}\left[\mathds{1}_{\{|V_i-\mathbb{E}(V_i)|>\epsilon\sigma_n\}}\right]}\\
&=&\sum_{i=1}^r\sqrt{\mathbb{E}\left[(V_i-\mathbb{E}(V_i))^4\right]} \sqrt{\mathbb{P}\left[|V_i-\mathbb{E}(V_i)|>\epsilon\sigma_n\right]}\\
&\stackrel{Chebychev}{\leq}& \sum_{i=1}^r \sqrt{\mathbb{E}\left[(V_i-\mathbb{E}(V_i))^4\right]}\frac{\sigma_i}{\epsilon s_n}\\
&=&2\sqrt{3}\frac{\kappa^2}{\epsilon}\sum_{i=1}^r \lambda_i^4 \sqrt{(1+2\delta_i^2)^2+4(1+4\delta_i^2)}\frac{\sigma_i}{s_n}\,.
\end{eqnarray*}

In using
\[(1+2\delta_i^2)^2+4(1+4\delta_i^2)=(5+2\delta_i^2)^2-20\le(5+2\delta_i^2)^2\]
we get with $\sigma_{\max}=\sup_{i} \sigma_i$ the following inequality
\begin{eqnarray*}
&&\frac{1}{s_n^2}\sum_{i=1}^r \mathbb{E}\left[(V_i-\mathbb{E}(V_i))^2\mathds{1}_{\{|V_i-\mathbb{E}(V_i)|>\epsilon s_n\}}\right]\le 2\sqrt{3}\frac{\kappa^2}{\epsilon}\frac{\sigma_{\max}}{s_n}\frac{1}{s_n^2}\sum_{i=1}^r \lambda_i^4 (5+2\delta_i^2)\\
&=&\frac{\sqrt{3}}{\epsilon}\frac{\sigma_{\max}}{s_n}\frac{5 tr(\bSigma^4)+2\kappa^{-1}\bmu_a^T\bSigma^3\bmu_a}{tr(\bSigma^4)+2\kappa^{-1}\bmu_a^T\bSigma^3\bmu_a}
\le\frac{5\sqrt{3}}{\epsilon}\frac{\sigma_{\max}}{s_n}\,.
\end{eqnarray*}

Using
\begin{eqnarray*}
(\bm u_i^T\bmu_a)^2&=&\left(\bm u_i^T\bmu+\frac{a_2}{a_1} \bm u_i^T\bm m\right)^2
=2(\bm u_i^T\bmu)^2+2\left(\frac{a_2}{a_1}\bm u_i^T\bm m\right)^2\\
&=&2L_2^2\left( 1+\left(\frac{a_2}{a_1}\right)^2\right)<\infty
\end{eqnarray*}
and Assumptions (A1) and (A2), we get
\begin{eqnarray*}
\frac{\sigma_{\max}^2}{\sigma_n^2}&=&\frac{\sup_i(\lambda_i^4 (1+2\delta_i^2))}{tr(\bSigma^4)+2\kappa^{-1}\bmu_a^T\bSigma^3\bmu_a}=\frac{\sup_i(\lambda_i^4 + 2\kappa^{-1}\lambda_i^3 (\bm u_i^T\bmu_a)^2)}{tr(\bSigma^4)+2\kappa^{-1}\bmu_a^T\bSigma^3\bmu_a}\longrightarrow 0
\end{eqnarray*}
which verifies the Linderberg condition.

Since
\begin{eqnarray*}
\sum_{i=1}^r\mathbb{E}(V_i)&=&\kappa\sum_{i=1}^r\lambda_i^2\mathbb{E}\left(\zeta_i\right)=\kappa\sum_{i=1}^r\lambda_i^2\left(1+\delta_i^2\right)
=\kappa tr(\bSigma^2)+\bmu_a^T\bSigma\bmu_a
\end{eqnarray*}
we obtain by using the Linderberg's central limit theorem that
\begin{eqnarray*}
\sqrt{\frac{1}{\kappa}}\frac{\tilde{\mathbf z}^T \bSigma\tilde{\mathbf z}- \kappa  tr(\bSigma^2)-\bmu_a^T\bSigma\bmu_a}{\sqrt{2\kappa tr(\bSigma^4)+4\bmu_a^T\bSigma^3\bmu_a}}
&\stackrel{d}{\longrightarrow} \mathcal N (0,1)
\end{eqnarray*}

Let $\mathbf a=(a_1,2a_2)^T$. Then the last identity leads to
\begin{align}\label{th5-Eq5}
&\sqrt{n}\left[
\begin{array}{ccc}
 \mathbf a^T\left(\begin{array}{c}
 \mathbf z^T \bSigma\mathbf z\\
  \mathbf m^T \bSigma\mathbf z
  \end{array}\right)&
  -&\mathbf a^T\left(
  \begin{array}{c}
  \kappa tr(\bSigma^2)+ \bmu^T \bSigma\bmu\\
  \mathbf m^T \bSigma\bmu
  \end{array}
  \right)
\end{array}
\right]\nonumber\\
&\stackrel{d}{\longrightarrow}
\mathcal N\left(
0,\mathbf a^T \frac{\kappa}{c}\left(\begin{array}{cc}
2\kappa tr(\bSigma^4)+4 \bmu\bSigma^3\bmu&
2\mathbf m^T\bSigma^3\bmu
\\
2\mathbf m^T\bSigma^3\bmu&
\mathbf m^T\bSigma^3\mathbf m
\end{array}
\right) \mathbf a\right)
\end{align}
which implies that the vector
$\left( \mathbf z^T \bSigma\mathbf z,\mathbf m^T \bSigma\mathbf z\right)^T$ is asymptotically multivariate normally distributed because $\mathbf a$ is an arbitrary fixed vector.

Taking into account to \eqref{th5-Eq1},\eqref{th5-Eq5} and the fact that $\zeta, z_0$, and $\mathbf z$ are mutually independent we get the following asymptotic result
\begin{eqnarray*}
&&\sqrt{n}\left[
\begin{array}{ccc}
\left(\begin{array}{c}
\frac{\zeta}{n}\\
 \mathbf z^T \bSigma\mathbf z\\
  \mathbf m^T \bSigma\mathbf z\\
  \frac{z_0}{\sqrt{n}}
  \end{array}\right)
  &-&
  \left(
  \begin{array}{c}
  1\\
  \kappa tr(\bSigma^2)+\bmu^T \bSigma\bmu\\
  \mathbf m^T \bSigma\bmu\\
  0
  \end{array}
  \right)
\end{array}
\right]\\
&&\stackrel{d}{\longrightarrow}\mathcal N
\left(
0,\left(\begin{array}{cccc}
2& 0& 0& 0\\
0& 2\frac{\kappa^2}{c} tr(\bSigma^4)+4\frac{\kappa}{c}\bmu^T\bSigma^3\bmu & 2\frac{\kappa}{c}\mathbf m^T\bSigma^3\bmu& 0\\
0& 2\frac{\kappa}{c}\mathbf m^T\bSigma^3\bmu & \frac{\kappa}{c}\mathbf m^T\bSigma^3\mathbf m& 0\\
0&0&0&1\\
\end{array}
\right)\right)
\end{eqnarray*}

Finally, the application of the delta method leads to
\[\sqrt{n}\sigma^{-1}\left(\frac{1}{n}\mathbf m ^T \mathbf {Az}- \mathbf m ^T \bSigma\bmu\right) \stackrel{d}{\longrightarrow}\mathcal N \left(0,1\right), \]
where
{
\begin{eqnarray*}
\sigma^2&=&
\left(
\begin{array}{cccc}
\mathbf m^T \bSigma\bmu&0&1&\left[\left[\kappa tr(\bSigma^2)+\bmu^T \bSigma\bmu\right]\mathbf m^T\bSigma\mathbf{m}-\left(\mathbf m^T\bSigma\bmu\right)^2\right]^{\frac{1}{2}}\\
\end{array}
\right)
\\
&\times&\left(\begin{array}{cccc}
2& 0& 0& 0\\
0& 2\frac{\kappa^2}{c} tr(\bSigma^4)+4\frac{\kappa}{c}\bmu^T\bSigma^3\bmu & 2\frac{\kappa}{c}\mathbf m^T\bSigma^3\bmu& 0\\
0& 2\frac{\kappa}{c}\mathbf m^T\bSigma^3\bmu & \frac{\kappa}{c}\mathbf m^T\bSigma^3\mathbf m& 0\\
0&0&0&1\\
\end{array}
\right)\\
&\times&
\left(
\begin{array}{c}
\mathbf m^T \bSigma\bmu\\
0\\
1\\
\left[\left[\kappa tr(\bSigma^2)+\bmu^T \bSigma\bmu\right]\mathbf m^T\bSigma\mathbf{m}-\left(\mathbf m^T\bSigma\bmu\right)^2\right]^{\frac{1}{2}}
\end{array}\right)\\
&=& \left(\mathbf m^T \bSigma\bmu\right)^2+\mathbf m^T\bSigma\mathbf{m}
\left[\kappa tr(\bSigma^2)+\bmu^T \bSigma\bmu\right]+\frac{\kappa}{c}\mathbf m^T\bSigma^3\mathbf{m}.
\end{eqnarray*}
}
\end{proof}

Finally, we extend the results of Theorem \ref{th5} to the case of finite number of linear combinations of the elements of $\mathbf{Az}$. The results are summarised in the following theorem.

\begin{theorem}\label{th6}
Let $\mathbf A \sim \mathcal W _{k} (n, \bSigma)$ with $rank (\bSigma)=r < p$ and let $\mathbf z \sim \mathcal N _k (\bmu, \kappa \bSigma)$,$\kappa>0$. Assume $\frac{r}{n}=c+o(n^{-1/2}), c\in [0,+\infty)$ and $\kappa r=O(1)$ as $n\rightarrow \infty$.Let $\mathbf M = (\mathbf{m}_1,\ldots,\mathbf{m}_p)^T : p \times k$ be a matrix of constants of rank $p<r\le n$ with probability one and let $|\bm u_i^T\mathbf m_j|\leq L_2$ for all $i=1,\ldots,r$ and $j=1,\ldots,p$ uniformly on $k$. Assume that $\mathbf A$ and $\mathbf z$ are independently distributed. Then under (A1) and (A2) the asymptotic distribution of $\mathbf{MAz}$ under the double asymptotic regime is given by
\begin{eqnarray*}
\sqrt{n}\mathbf{\Omega}^{-1/2}\left(\frac{1}{n}\mathbf M\mathbf{Az}-\mathbf M\bSigma\mathbf z\right)&\stackrel{d}{\longrightarrow}\mathcal N_p
\left(\mathbf{0},\mathbf{I}_p\right)
\end{eqnarray*}
where
\begin{eqnarray}
\mathbf{\Omega}=\mathbf M\bSigma\bmu\bmu^T\bSigma\mathbf M^T+\mathbf M\bSigma\mathbf M^T\left[\kappa tr(\bSigma^2)+\bmu^T\bSigma\bmu\right]+ \frac{\kappa}{c}\mathbf M\bSigma^3\mathbf{M}^T.
\end{eqnarray}
\end{theorem}

\begin{proof}
For all $\mathbf{l}\in \mathbb{R}^p$-fixed we consider $\mathbf{l}^T\mathbf{MAz}$. The rest of the proof follows from Theorem \ref{th5} with $\mathbf{m}=\mathbf{M}^T \mathbf{l}$ and the fact that $\mathbf{l}$ is an arbitrary vector.
\end{proof}

\section{Finite sample performance}
In this section we present the results of a Monte Carlo simulation study where the performance of the obtained asymptotic distribution for the product of a singular Wishart matrix and a singular Gaussian vector is investigated.

In our simulation we fix $\mathbf{m}=\mathbf{1}/k$ where $\mathbf{1}$ denotes the $k$-dimensional vector of ones and generated each element of $\bmu$ from the uniform distribution on $[-1,1]$. The population covariance matrix was drawn in the following way:
\begin{itemize}
\item $r$ non-zero eigenvalues of $\bSigma$ were generated from the uniform distribution on $(0,1)$ and the rest were set to be zero;
\item the eigenvectors were generated form the Haar distribution by simulating a Wishart matrix with identity covariance matrix and calculating its eigenvectors.
\end{itemize}
Both the mean vector and the population covariance matrix obtained by such setting satisfy the assumptions (A1) and (A2).

We compare the asymptotic density of the standardized random variable $\mathbf m^T\mathbf{Az}$ with its finite-sample one which is obtained by applying the stochastic representation of Corollary \ref{c2}. More precisely, we draw $N=10^4$ independent realizations of the standardized random variable $\mathbf m^T\mathbf{Az}$ by using the following algorithm
\begin{description}
\item [a)] generate $\mathbf m^T\mathbf{Az}$ by using stochastic representation \eqref{c2_sp} of Corollary \ref{c2} expressed as
\begin{equation*}
\mathbf m ^T \mathbf {Az}
\stackrel{d}{=}
\zeta \mathbf m^T \bSigma \mathbf z  +\sqrt{\zeta}
\left[\mathbf z^T \bSigma \mathbf z \cdot \mathbf m^T \bSigma \mathbf m - (\mathbf m^T \bSigma \mathbf z)^2\right]^{1/2} z_0,
\end{equation*}
where $\zeta \sim \chi^2_n$, $z_0 \sim \mathcal N (0,1)$, $\mathbf z \sim \mathcal N _k (\bmu, \kappa \bSigma)$; $\zeta$, $z_0$, and $\mathbf z$ are mutually independent.
\item [b)] compute
\begin{eqnarray*}\sqrt{n}\sigma^{-1}\left(\frac{1}{n}\mathbf m ^T \mathbf {Az}- \mathbf m ^T \mathbf {\Sigma}\bmu\right)
\end{eqnarray*}
 where
 \begin{eqnarray*}
 \sigma^2= \left(\mathbf m^T \bSigma\bmu\right)^2+\mathbf m^T\bSigma\mathbf{m}
\left[\kappa tr(\bSigma^2)+\bmu^T \bSigma\bmu\right]+\frac{\kappa}{c}\mathbf m^T\bSigma^3\mathbf{m}.
 \end{eqnarray*}
\item [c)] repeat a)-b) $N$ times.
\end{description}
Then, the elements of the generated sample are used to construct a kernel density estimator which is compared to the asymptotic distribution, i.e. to the density of the standard normal distribution. As a kernel, we make use of Epanechnikov kernel with the bandwidth chosen by applying the Silverman's rule of thumb.

The results of the simulation study are summarized in Figure \ref{fig1} for $n=500$ and in Figure \ref{fig2} for $n=1000$. In both cases we set $\kappa=1/n$. Finally, $k=750$ is chosen for $n=500$ and $k \in\{750, 990\}$ for $n=1000$. Furthermore, several values of $r$ are considered such that $c=\{0.1,0.5,0.8,0.95\}$. The finite sample distributions are shown as dashed line, while the asymptotic distributions are solid lines. All obtained results show a good performance of the asymptotic approximation which is almost indistinguishable from the corresponding finite sample density. This result continues to be true even in the extreme case of $c=0.95$.

\begin{figure}
\subfigure[$n=500,c=0.1,k=750$]{\includegraphics[width = 3.5in]{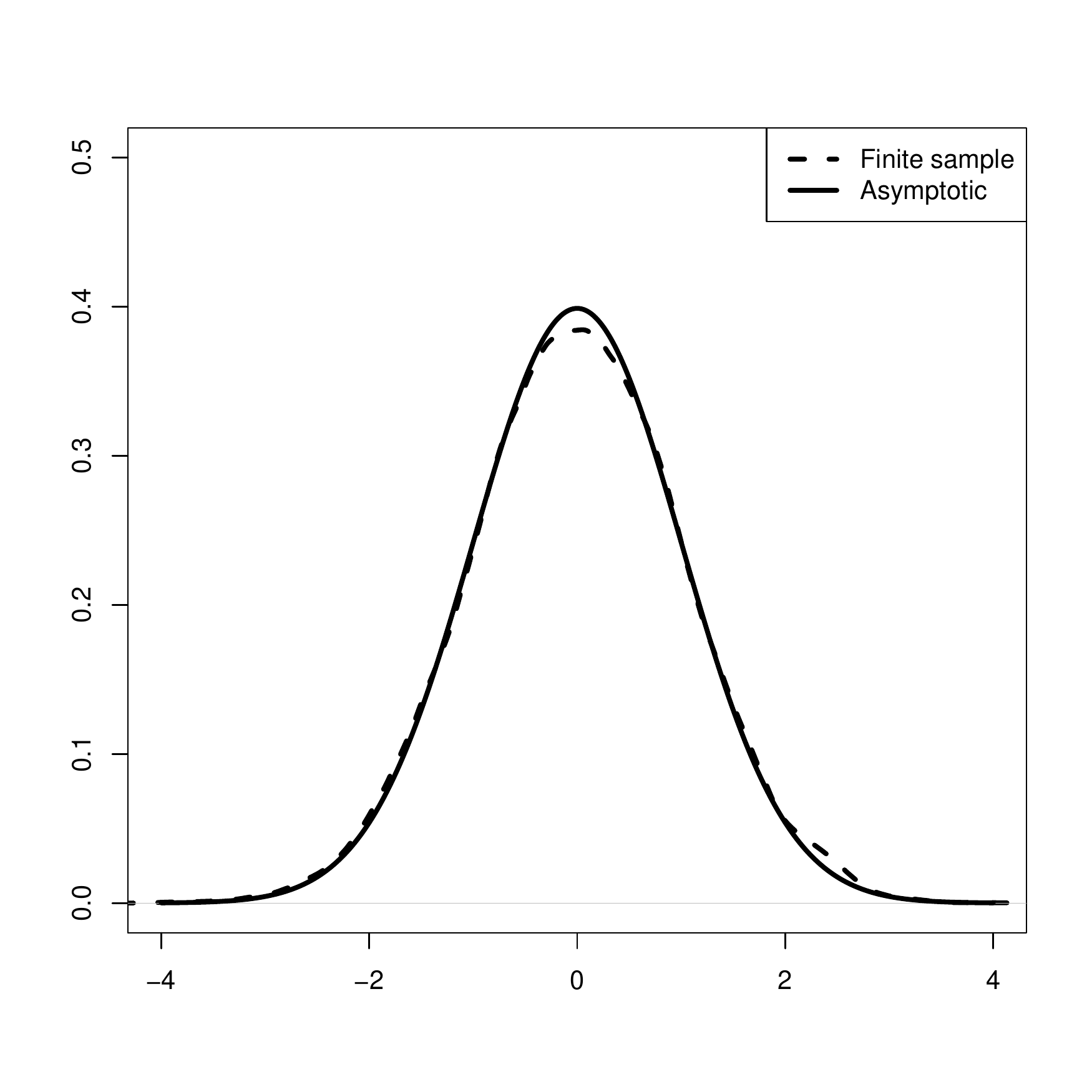}}
\subfigure[$n=500,c=0.5,k=750$]{\includegraphics[width = 3.5in]{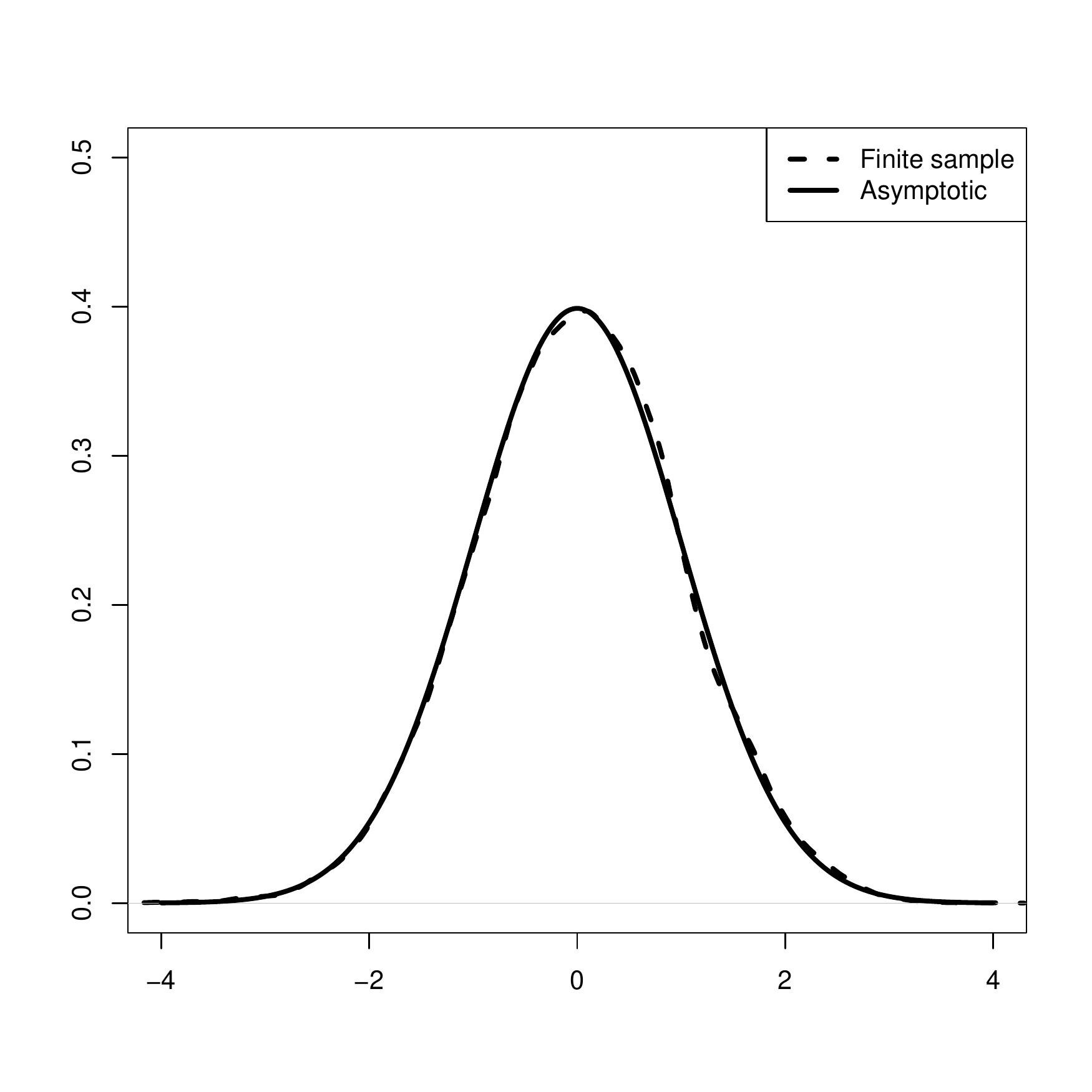}}
\subfigure[$n=500,c=0.8,k=750$]{\includegraphics[width = 3.5in]{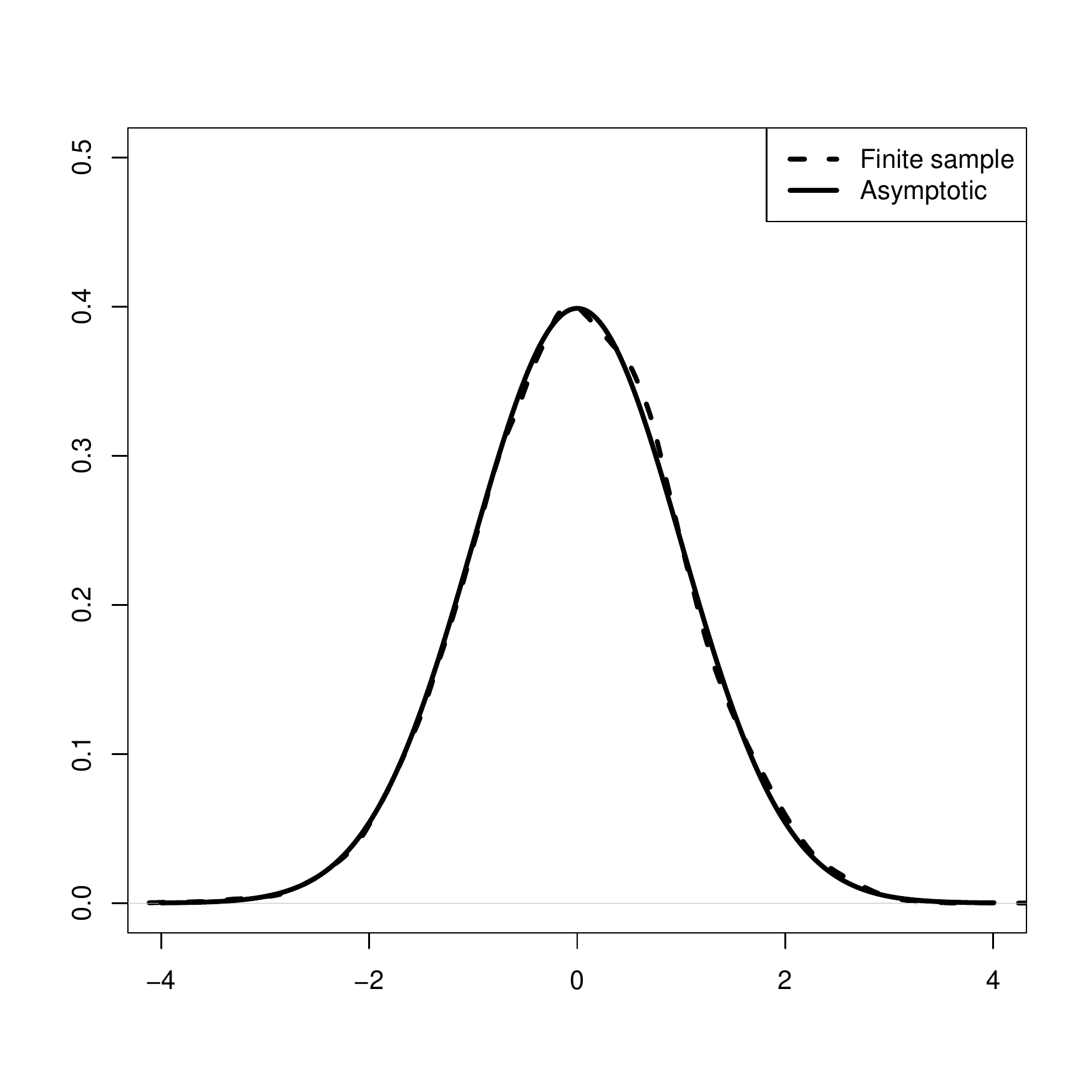}}
\subfigure[$n=500,c=0.95,k=750$]{\includegraphics[width = 3.5in]{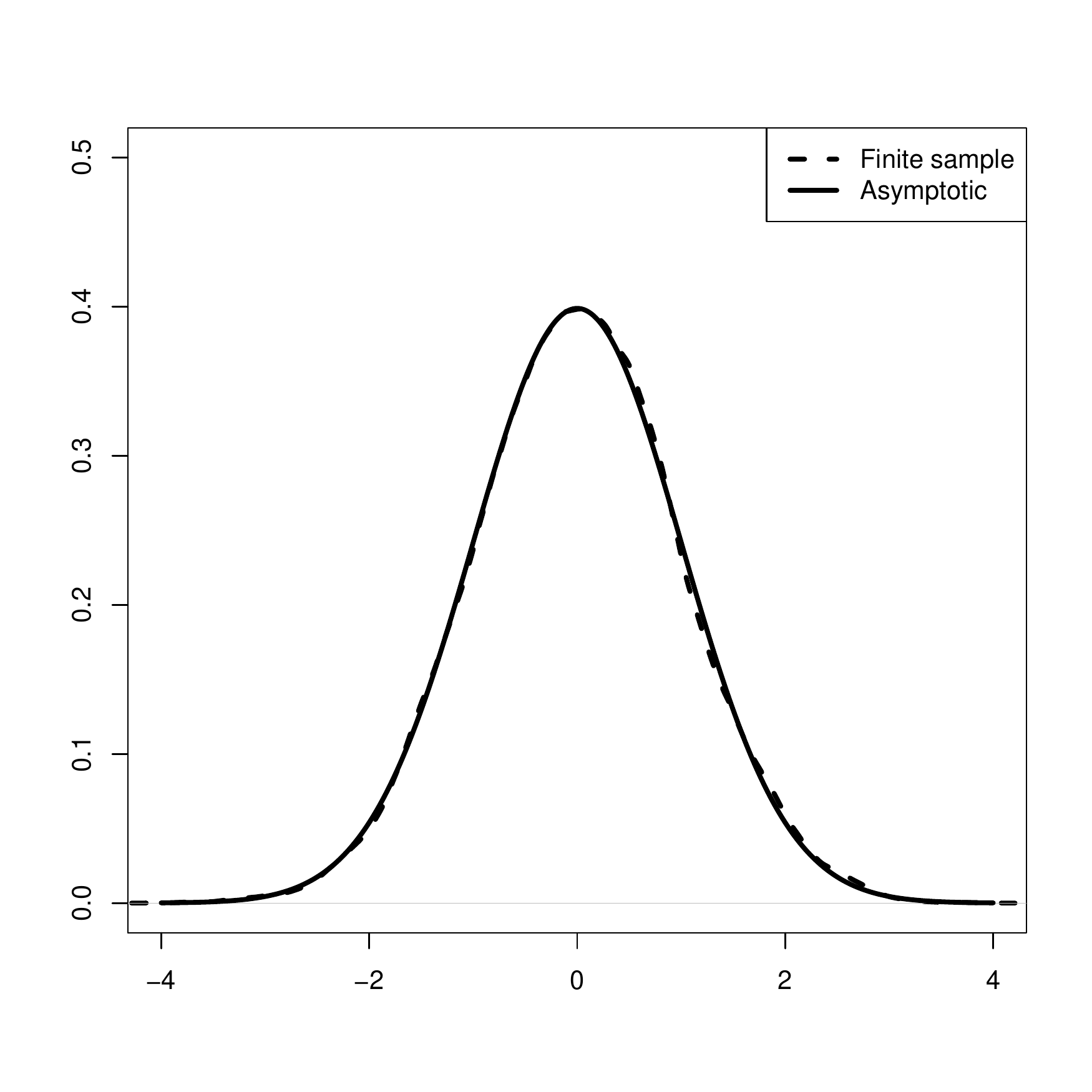}}
\caption{Asymptotic distribution and the kernel density estimator of the finite sample distribution calculated for the product of a singular Wishart matrix and a singular normal vector ($n=500$).}
\label{fig1}
\end{figure}

\begin{figure}
\subfigure[$n=1000,c=0.1,k=750$]{\includegraphics[width = 3.5in]{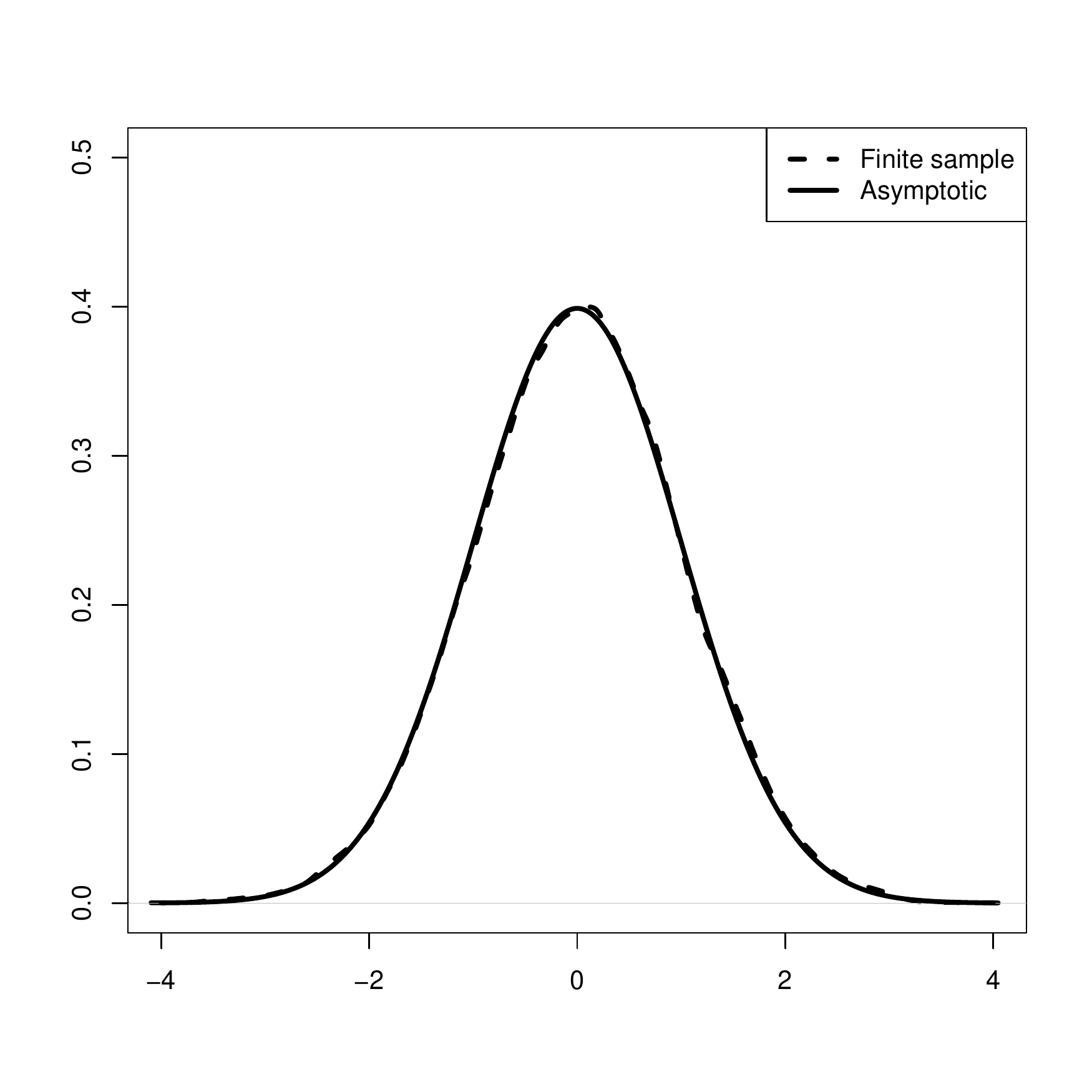}}
\subfigure[$n=1000,c=0.5,k=750$]{\includegraphics[width = 3.5in]{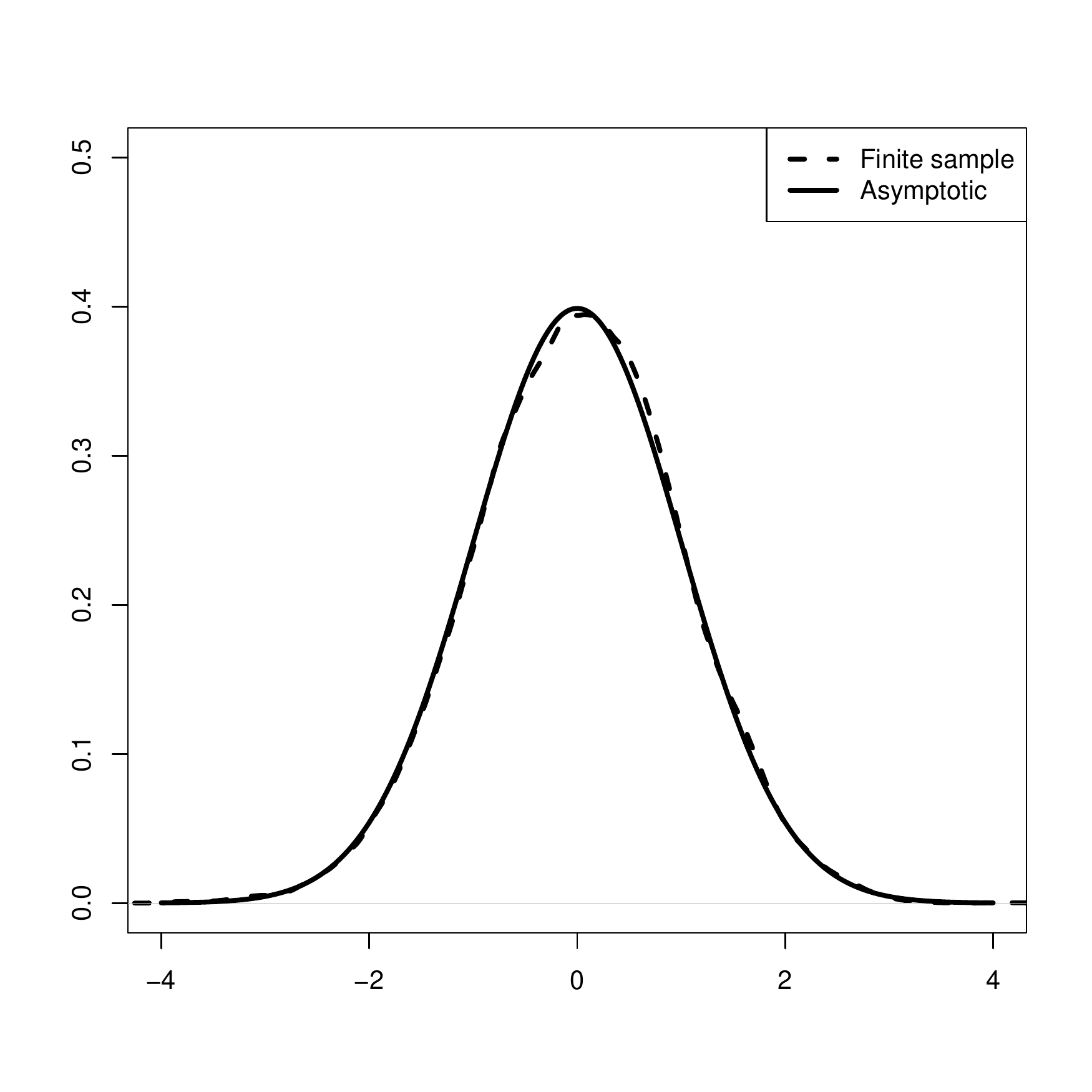}}
\subfigure[$n=1000,c=0.8,k=990$]{\includegraphics[width = 3.5in]{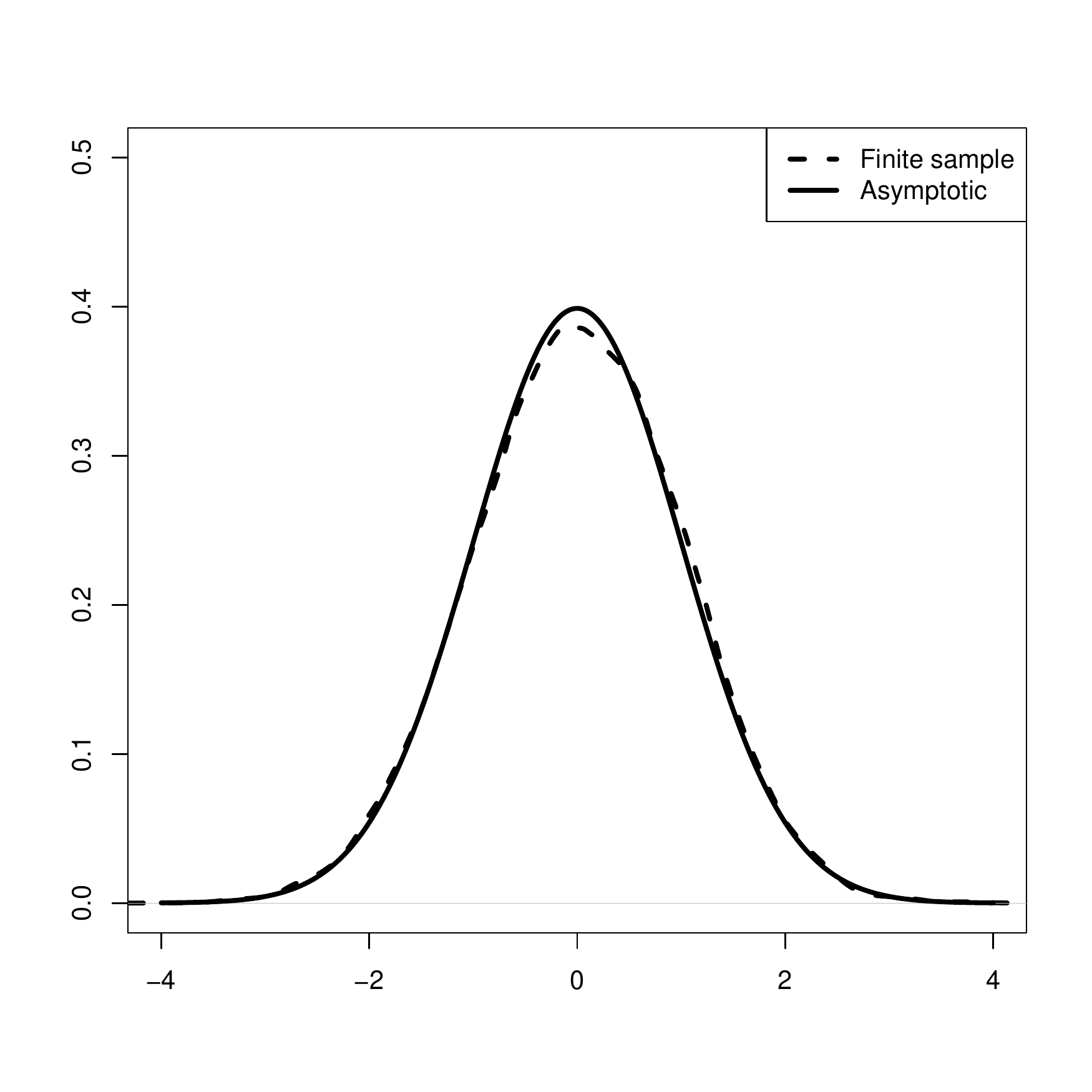}}
\subfigure[$n=1000,c=0.95,k=990$]{\includegraphics[width = 3.5in]{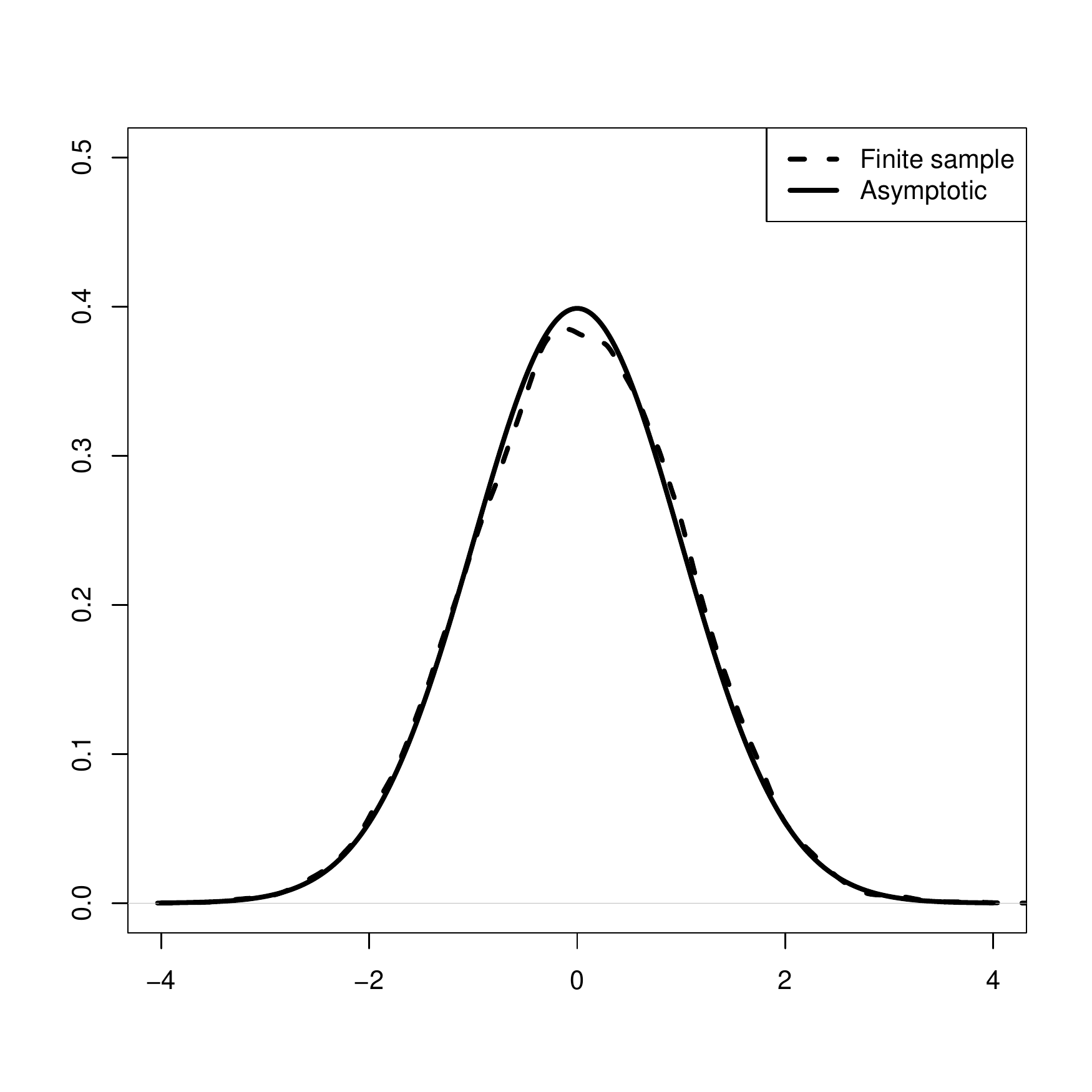}}
\caption{Asymptotic distribution and the kernel density estimator of the finite sample distribution calculated for the product of a singular Wishart matrix and a singular normal vector ($n=1000$).}
\label{fig2}
\end{figure}

\section{Summary}
Wishart distribution and normal distribution are widely spread in both statistics and probability theory with numerous and useful applications in finance, economics, environmental sceinces, biology, etc. Different functions involving a Wishart matrix and a normal vector have been studied in statistical literature recently. However, to the best of our knowledge, combinations of a singular Wishart matrix and a singular normal vector have not been investigated up to now.

In this paper we analyse the product of a singular Wishart matrix and a singular Gaussian vector. A very useful stochastic representation of this product is obtained, which is later used to derive its characteristic function as well as to provide an efficient way how the elements of the product could be simulated in practice. With the use of the derived stochastic representation there is no need in generating a large dimensional Wishart matrix. Its application speeds up simulation studies where the product of a singular Wishart matix and a singular normal vector is present. Furthermore, we prove the asymptotic normality of the product under the double asymptotic regime. In a numerical study, a good performance of the obtained asymptotic distribution is documented. Even in the extreme case of $c=0.95$ it produces a very good approximation of the corresponding finite sample distribution obtained by applying the derived stochastic representation.

\bibliography{GIW}
\end{document}